\documentclass{amsart}
\usepackage{amsmath,amssymb,amscd,amsthm,verbatim,alltt,amsfonts,array}
\usepackage[english]{babel}
\usepackage{latexsym}
\usepackage{amssymb}
\usepackage{euscript}
\usepackage{graphicx}
\usepackage[all]{xy}
\usepackage{hyperref}

\newtheorem{theorem}{Theorem}
\theoremstyle{plain}

\newtheorem{corollary}{Corollary}

\newtheorem{lemma}{Lemma}

\newtheorem{proposition}{Proposition}

\numberwithin{equation}{section}

\def\ox{\otimes}
\def\M#1{\mathcal M({#1})}

\begin{document}

\title{Firm Frobenius monads and firm Frobenius algebras}

\author{Gabriella B\"ohm}
\address{Wigner Research Centre for Physics, Budapest,
H-1525 Budapest 114, P.O.B.\ 49, Hungary}
\email{\tt bohm.gabriella@wigner.mta.hu}
\author{Jos\'e G\'omez-Torrecillas}
\address{Departamento de \'Algebra,
Universidad de Granada,
E-18071 Granada, Spain}
\email{\tt gomezj@ugr.es}
\date{Apr 2013}

\markboth{Author(s)}{Firm Frobenius monads}

\maketitle 

\begin{center}
{\em Dedicated to\\
Toma Albu and Constantin N\u ast\u asescu on the occasion of their 70th
birthdays} 
\end{center}
\bigskip

\begin{abstract}
{\em Firm Frobenius algebras} are firm algebras and counital coalgebras such
that the comultiplication is a bimodule map. They are investigated by
categorical methods based on a study of adjunctions and lifted functors. Their
categories of comodules and of firm modules are shown to be isomorphic if and
only if a canonical comparison functor from the category of comodules to the
category of non-unital modules factorizes through the category of firm
modules. This happens for example if the underlying algebra possesses local
units, e.g. the firm Frobenius algebra arises from a co-Frobenius coalgebra
over a base field; or if the comultiplication splits the multiplication (hence
the underlying coalgebra is coseparable). 
\end{abstract}

\begin{quotation}
\noindent{\bf Key Words}: {firm Frobenius monad, firm Frobenius adjunction,
firm module, comodule, separable functor} 

\noindent{\bf 2010 Mathematics Subject Classification}: { Primary 16L60, 
Se\-con\-dary 16D90, 18C15, 18C20, 18A40, 16T15}
\end{quotation}


\thispagestyle{empty} 

\section{Introduction}
The classical notion of {\em Frobenius algebra} due to Brauer and Nesbitt
\cite{BraNes} can be reformulated in terms of the existence of a suitable
coalgebra structure on the algebra \cite{LAbrams}. Thus, a Frobenius algebra
over a commutative ring $k$ is a $k$-module carrying the structures of an
associative and unital algebra and a coassociative and counital
coalgebra. These structures are required to be compatible in the sense that
the comultiplication is a bimodule map (with respect to the actions provided 
by the multiplication). Equivalently, the multiplication is a bicomodule map
(with respect to the coactions provided by the comultiplication). As discussed
by Abrams in \cite{LAbrams}, this compatibility condition results in an
isomorphism between the category of modules and the category of comodules over
a Frobenius algebra. 

In \cite{Street:Frob}, Frobenius algebras were treated by Street in the
broader framework of monoidal (bi)categories. The behavior of the module and
comodule categories was given a deep conceptual explanation and a number of
equivalent characterizations of Frobenius monoids was given. Applying it to
the monoidal category of functors $\mathbb A\to \mathbb A$ for an arbitrary
category $\mathbb A$, the notion of {\em Frobenius monad} is obtained. By
\cite[Theorem 1.6]{Street:Frob}, a Frobenius monad is a monad $M:\mathbb A\to
\mathbb A$, with multiplication $\mu:M^2\to M$ and unit $\eta: \mathbb A\to M$
such that any of the following equivalent assertions holds. 
\begin{itemize}
\item[{(i)}] 
There exist natural transformations $\varepsilon:M\to \mathbb A$ and
$\varrho:\mathbb A \to M ^2$ such that $M\varepsilon\cdot M\mu\cdot \varrho
M=M=\varepsilon M\cdot \mu M \cdot M\varrho$. 
\item[{(ii)}] 
There exists a comonad structure $\delta:M\to M^2$, $\varepsilon:M\to \mathbb
A$ such that $\mu M\cdot M \delta =\delta.\mu=M\mu\cdot \delta M$. 
\item[{(iii)}] 
The forgetful functor from the category of $M$-modules to $\mathbb A$
possesses a right adjoint $A\stackrel \phi \to A'\mapsto (MA,\mu A) \stackrel
{M\phi} \longrightarrow (MA',\mu A')$. 
\end{itemize}
It is immediate from (ii) that an endofunctor $(-)\ox R$ on the category of
modules over a commutative ring $k$ is a Frobenius monad if and only if $R$ is
a Frobenius $k$-algebra. Characterization (i) yields a description of
Frobenius algebras in terms of a {\em functional} $\varepsilon_k:R\to k$ and a
{\em Casimir element} $\varrho_k(1)\in R\ox R$. 

The above classical notion of Frobenius algebra is essentially self-dual: the
algebra and coalgebra structures play symmetric roles. So if allowing the
algebra to be non-unital, it is not immediately clear what properties remain
true. 

Our approach to non-unital Frobenius algebras in this paper is based on
Street's categorical treatment in \cite{Street:Frob}. Generalizing non-unital
algebras, we start with discussing {\em non-unital monads}; that is
endofunctors equipped with an associative multiplication possibly without a
unit. While there is an evident notion of their {\em non-unital modules}, our 
definition of a {\em firm module} is slightly more sophisticated. It leads to
the notion of a {\em firm monad} which is a non-unital monad whose free
modules are firm. 

A {\em non-unital Frobenius monad} is defined as a (coassociative and
counital) co\-monad equipped also with an associative but not necessarily
unital multiplication satisfying the compatibility conditions in (ii) above. 
Associated to it, there is the usual Eilenberg-Moore category of
(coassociative and counital) comodules over the constituent comonad and the
categories of non-unital, and of firm modules over the constituent non-unital
monad. Generalizing the equivalence (ii)$\Leftrightarrow$(iii) above, we
investigate in what sense the corresponding forgetful functors possess
adjoints. 

This analysis leads in particular to a canonical comparison functor from the
category of comodules to the category of non-unital modules. It is proven to
factorize through the category of firm modules if and only if the underlying
monad is firm and the category of firm modules and the category of comodules
are isomorphic. We collect situations when this happens. In particular,
we show that any adjunction in which the left adjoint is separable, 
induces a firm Frobenius monad. For firm Frobenius monads of this kind, the
category of firm modules and the category of comodules are shown to be
isomorphic. 

We apply our results to algebras over commutative rings. This yields a
generalization of Abrams' theorem \cite{LAbrams} to {\em firm} Frobenius
algebras. In particular, the isomorphism between the category of firm modules
and the category of comodules follows from our theory in the following
situations. 
\begin{itemize}
\item For firm Frobenius algebras arising from coseparable coalgebras (and
even from coseparable corings) over any base ring. This provides an
alternative proof of \cite[Proposition 2.17]{BoVe}. 
\item For firm Frobenius algebras with local units.
\item In particular, for firm Frobenius algebras arising from co-Frobenius
coalgebras over a field. This provides an alternative proof of \cite[Theorem
2.3]{CaDaNa} and \cite[Proposition 2.7]{Beattie/alt:1998}. 
\end{itemize}
A firm algebra $R$ with a {\em non-degenerate} multiplication is shown to be a
firm Frobenius algebra if and only if there exists a generalized Casimir
element in the multiplier algebra of $R\ox R$ (cf. conditions (i) above). 

\section{Non-unital monads and firm modules}

\subsection{Non-unital monad}
By a {\em non-unital monad} on a category $\mathbb A$ we mean a pair of a
functor $M:\mathbb A \to \mathbb A$ and a natural transformation $\mu:M^2\to
M$ obeying the associativity condition 
$$
\xymatrix@C=60pt@R=15pt{
M^3\ar[r]^-{M\mu}\ar[d]_-{\mu M}&
M^2\ar[d]^-\mu\\
M^2\ar[r]^-\mu&
M.}
$$

\subsection{Non-unital module}\label{sec:nu_module}
By a {\em non-unital module} over a non-unital monad $M:\mathbb A \to \mathbb
A$ we mean a pair of an object $A$ and a morphism $\alpha:MA\to A$ (called the
$M$-{\em action}) in $\mathbb A$ obeying the associativity condition 
\begin{equation}\label{ass}
\xymatrix@C=60pt@R=15pt{
M^2A\ar[r]^-{M\alpha}\ar[d]_-{\mu A}&
MA\ar[d]^-\alpha\\
MA\ar[r]^-\alpha&
A.}
\end{equation}
Morphisms of non-unital $M$-modules are morphisms $A\to A'$ in $\mathbb A$
which are compatible with the $M$-actions in the evident sense. These objects
and morphisms define the category $\mathbb A_M$ of non-unital $M$-modules. 

In terms of the forgetful functor $U_M:\mathbb A_M\to \mathbb A$ and the
functor 
$$
F_M:\mathbb A \to \mathbb A_M,\qquad
A\stackrel{\phi}\to A' \mapsto 
(MA,\mu A) \stackrel{M\phi}\longrightarrow (MA',\mu A'),
$$
we can write $M=U_M F_M$. Although there is a natural transformation 
$$
\alpha: F_M U_M (A,\alpha)\to (A,\alpha),
$$
in the absence of a unit this does not provide an adjunction.

\subsection{Firm module}
We say that a non-unital module $(A,\alpha)$ over a non-unital monad
$M:\mathbb A \to \mathbb A$ is {\em firm} if $\alpha$ is an epimorphism in
$\mathbb A$ and the fork 
\begin{equation}\label{eq:firm_module}
\xymatrix@C=40pt{
M^2A \ar@<2pt>[r]^-{\mu A} \ar@<-2pt>[r]_-{M\alpha}&
MA\ar[r]^-\alpha&
A}
\end{equation}
is a coequalizer in $\mathbb A_M$. The full subcategory of firm modules in
$\mathbb A_M$ will be denoted by $\mathbb A_{(M)}$, and we will denote by $J$
the full embedding $\mathbb A_{(M)}\to \mathbb A_M$. When $M$ is a usual (unital)
monad, then the category of firm $M$-modules is just the usual Eilenberg-Moore
category of unital $M$-modules.

If the functor underlying a non-unital monad is known to preserve
epimorphisms, then a simpler criterion for firmness can be given. 

\begin{lemma}\label{lem:Malpha}
Let $M:\mathbb A\to \mathbb A$ be a non-unital monad and $(A,\alpha)$ be a
non-unital $M$-module. If \eqref{eq:firm_module} is a coequalizer in $\mathbb
A$ and $M\alpha$ is an epimorphism in $\mathbb A$, then $(A,\alpha)$ is a firm
$M$-module.  
\end{lemma}

\begin{proof}
By assumption, $\alpha$ is a (regular) epimorphism in $\mathbb A$. Consider a
morphism $\kappa$ of non-unital $M$-modules from $(MA, \mu A)$ to any
non-unital $M$-module $(X,\xi)$ such that $\kappa\cdot \mu A= \kappa\cdot
M\alpha$. Since \eqref{eq:firm_module} is a coequalizer in $\mathbb A$, there
is a unique morphism $\tilde \kappa:A \to X$ in $\mathbb A$ satisfying
$\tilde\kappa\cdot \alpha=\kappa$. The subdiagrams of 
$$
\xymatrix@R=15pt@C=40pt{
MA\ar[dd]_-{M\tilde\kappa}\ar[rrr]^-\alpha&&&
A\ar[dd]^-{\tilde\kappa}\\
&M^2A\ar[lu]^-{M\alpha}\ar[ld]_-{M\kappa}\ar[r]^-{\mu A}&
MA\ar[ru]_-{\alpha}\ar[rd]^-\kappa\\
MX\ar[rrr]^-\xi&&&
X}
$$
commute since $\kappa$ is a morphism of non-unital $M$-modules and $\alpha$ is
associative. So by the assumption that $M\alpha$ is an epimorphism in $\mathbb
A$, also the exterior commutes proving that $\tilde \kappa$ is a morphism in
$\mathbb A_M$ hence \eqref{eq:firm_module} is a coequalizer in $\mathbb A_M$. 
\end{proof}

\subsection{Firm monad} \label{sec:firm_monad}
We say that a non-unital monad $M$ on a category $\mathbb A$ is a {\em firm
 monad} if the functor $F_M:\mathbb A\to \mathbb A_M$ in Section
\ref{sec:nu_module} factorizes through some functor $F_{(M)}:\mathbb A\to
\mathbb A_{(M)}$ via the inclusion $J:\mathbb A_{(M)}\to \mathbb A_M$. That
is, for any object $A$ of $\mathbb A$, $\mu A$ is an epimorphism in $\mathbb
A$ and 
\begin{equation}\label{eq:firm_monad}
\xymatrix@C=40pt{
M^3 A\ar@<2pt>[r]^-{\mu MA} \ar@<-2pt>[r]_-{M\mu A}&
M^2A\ar[r]^-{\mu A}&
MA}
\end{equation}
is a coequalizer in $\mathbb A_M$. Then in terms of the forgetful functor
$U_{(M)}:\mathbb A_{(M)}\to \mathbb A$, the equality $M=U_{(M)} F_{(M)}$ holds
and there is a natural transformation 
$$
\alpha: F_{(M)} U_{(M)} (A,\alpha)=(MA,\mu A)\to (A,\alpha).
$$
However, in general this does not extend to an adjunction. 

\section{Non-unital monads versus adjunctions}

\subsection{Non-unital adjunction}\label{sec:nu_adj}
By a {\em non-unital adjunction} we mean a pair of functors $U:\mathbb B \to
\mathbb A$ and $F:\mathbb A \to \mathbb B$ together with a natural
transformation $\varphi:FU\to \mathbb B$. 

Associated to any non-unital adjunction $\varphi:FU\to \mathbb B$, there is a
non-unital monad $(UF,U\varphi F)$. Conversely, associated to any non-unital
monad $M:\mathbb A \to \mathbb A$, there is a non-unital adjunction $F_MU_M\to
\mathbb A_M$ as in Section \ref{sec:nu_module}. 

For any non-unital adjunction $\varphi:FU\to \mathbb B$, there is an induced
functor 
$$
{L}_{UF}: \mathbb B \to \mathbb A_{UF},\qquad 
B\stackrel\phi\to B'\mapsto
(UB,U\varphi B)\stackrel{U\phi}\longrightarrow (UB',U\varphi B').
$$

\subsection{Firm adjunction}\label{sec:firm_adj}
We say that a non-unital adjunction $\varphi:FU\to \mathbb B$ is {\em firm} if
the functor ${L}_{UF}: \mathbb B \to \mathbb A_{UF}$ in Section
\ref{sec:nu_adj} factorizes through some functor ${L}_{(UF)}: \mathbb B
\to \mathbb A_{(UF)}$ via the inclusion $J:\mathbb A_{(UF)}\to \mathbb
A_{UF}$. That is, for any object $B$ in $\mathbb B$, $U\varphi B$ is an
epimorphism in $\mathbb A$ and 
\begin{equation}\label{eq:firmadj}
\xymatrix@C=60pt{
UFUFUB \ar@<2pt>[r]^-{U\varphi F UB} \ar@<-2pt>[r]_-{UFU\varphi B}&
UFUB\ar[r]^-{U\varphi B}&
UB}
\end{equation}
is a coequalizer in $\mathbb A_{UF}$. 

For an adjunction $F\dashv U$ in the usual (unital and counital) sense,
\eqref{eq:firmadj} is a split coequalizer in $\mathbb A$ (in the sense of
\cite[page 93]{TTT}). Hence $F\dashv U$ is a firm adjunction by Lemma
\ref{lem:Malpha}. 

Associated to any firm adjunction $\varphi:FU\to \mathbb B$, there is a firm
monad $(UF,U\varphi F)$. Conversely, associated to any firm monad $M:\mathbb A
\to \mathbb A$, there is a firm adjunction $F_{(M)}U_{(M)}\to \mathbb A_{(M)}$
as in Section \ref{sec:firm_monad}. 

\section{Frobenius structures}

\subsection{Non-unital Frobenius monad}
By a {\em non-unital Frobenius monad} we mean a functor $M:\mathbb A \to
\mathbb A$ which carries a non-unital monad structure $\mu:M^2\to M$ and a
comonad structure $(\delta:M \to M^2,\varepsilon:M\to \mathbb A)$ such that
the following diagram commutes. 
\begin{equation}\label{Frob}
\xymatrix@R=10pt{
M^2\ar[rr]^-{M\delta}\ar[dd]_-{\delta M}\ar[rd]^-{\mu}&&
M^3\ar[dd]^-{\mu M}\\
& M \ar[rd]^-\delta\\
M^3\ar[rr]^-{M\mu}&&
M^2}
\end{equation}
 A {\em firm Frobenius monad} is a non-unital Frobenius monad which is a firm
monad. 

Let $M:\mathbb A \to \mathbb A$ be a non-unital Frobenius monad. As in the
case of any non-unital monad $M$, we denote by $\mathbb A_M$ the category of
non-unital $M$-modules and we denote by $\mathbb A_{(M)}$ the full subcategory
of firm $M$-modules. The usual Eilenberg-Moore category of counital comodules
for the comonad $M$ will be denoted by $\mathbb A^M$ with corresponding
forgetful functor $U^M:\mathbb A^M\to \mathbb A$ and its right adjoint $F^M:
A\stackrel \phi \to A' \mapsto (MA,\delta A)\stackrel {M\phi}\longrightarrow
(MA',\delta A')$. 

\subsection{Non-unital Frobenius adjunction}
We say that a (firm) non-unital adjunction $\varphi:FU\to \mathbb B$ is {\em
Frobenius} if $U$ is the left adjoint of $F$ (in the usual sense, with unit
$\eta:\mathbb B \to FU$ and counit $\varepsilon:UF \to \mathbb A$). Then $UF$
is a (firm) non-unital Frobenius monad, with multiplication $U\varphi F$,
comultiplication $U\eta F$ and counit $\varepsilon$. The aim of the next
sections is to prove the converse: to associate a (firm) non-unital Frobenius
adjunction to any (firm) non-unital Frobenius monad. 

\subsection{Firm Frobenius monads versus adjunctions to firm modules} 
\begin{proposition}\label{prop:firm_Frob_adj}
Any firm Frobenius monad $M:\mathbb A \to \mathbb A$ determines a firm
Frobenius adjunction $F_{(M)}U_{(M)}\to \mathbb A_{(M)}$ as in Section
\ref{sec:firm_monad} such that $U_{(M)}F_{(M)}=M$ as firm Frobenius
monads. 
\end{proposition}

\begin{proof}
The counit of the adjunction $U_{(M)}\dashv F_{(M)}$ is the counit
$\varepsilon:U_{(M)}F_{(M)}=M\to \mathbb A$ of the comonad $M$, and the unit
$\eta:\mathbb A_{(M)} \to F_{(M)}U_{(M)}$ is defined via universality of the
coequalizer (in $\mathbb A_{(M)}$) in the top row of 
$$
\xymatrix@C=40pt{
M^2A\ar@<2pt>[r]^-{\mu A}\ar@<-2pt>[r]_-{M\alpha}\ar[d]_-{\delta MA}&
MA\ar[r]^-\alpha\ar[d]^-{\delta A}&
A\ar@{-->}[d]^-{\eta (A,\alpha)}\\
M^3 A \ar@<2pt>[r]^-{M\mu A}\ar@<-2pt>[r]_-{M^2\alpha}&
M^2 A \ar[r]^-{M\alpha}&
MA,}
$$
for any firm $M$-module $(A,\alpha)$. The square on the left commutes serially
(in the sense of \cite[page 72]{TTT}) by the Frobenius condition \eqref{Frob}
and by naturality of $\delta$. The bottom row is a fork by the associativity
of $\alpha$. By the Frobenius property \eqref{Frob} of $M$, $\delta A$ is a
morphism of non-unital $M$-modules, hence so is $M\alpha \cdot \delta A$. This
proves the existence and the uniqueness of the morphism of non-unital
$M$-modules $\eta(A,\alpha):(A,\alpha)\to (MA,\mu A)$. 

Naturality of $\eta$ follows from commutativity of the diagram 
$$
\xymatrix@R=10pt{ 
A \ar^-{\eta(A,\alpha)}[rrr] \ar_-{\phi}[ddd] & & & 
MA \ar^-{M \phi}[ddd] \\ 
& MA \ar^-{\delta A}[r] \ar^-{\alpha}[lu] \ar^-{M \phi}[d] & 
M^2 A \ar_-{M \alpha}[ru] \ar^-{M^2 \phi}[d] & \\
& MA' \ar^-{\delta A'}[r] \ar_-{\alpha'}[dl] & M^2 A' \ar^-{M \alpha'}[dr]& \\
A' \ar^-{\eta(A',\alpha')}[rrr] & & & MA' }
$$
for any morphism of firm $M$-modules $\phi:(A,\alpha)\to (A',\alpha')$,
because $\alpha$ is an epimorphism (both in $\mathbb A_{(M)}$ and $\mathbb
A$). 

It remains to check that $\eta$ and $ \varepsilon $ satisfy the triangular
identities. Given an object $A$ of $\mathbb A$, $\eta F_{(M)} A =\eta(MA,\mu
A)$ is the unique morphism rendering commutative the diagram 
$$
\xymatrix@R=10pt@C=60pt
{M^2 A \ar^-{\mu A}[r] \ar_-{\delta M A}[d] & MA \ar^-{\eta F_{(M)} A}[d]\\
M^3A \ar^-{M \mu A}[r]& M^2 A.}
$$
So by the Frobenius condition \eqref{Frob}, $\eta F_{(M)} A =
\delta A$. Thus, 
$F_{(M)} \varepsilon \cdot \eta F_{(M)} = 
M \varepsilon \cdot \delta = F_{(M)}$.
Since $\alpha$ is an epimorphism in $\mathbb A$ for all $(A,\alpha)\in \mathbb
A_{(M)}$, the other triangle condition follows by the commutativity of 
$$
\xymatrix@C=45pt@R=10pt{
A= U_{(M)} (A, \alpha) \ar^-{U_{(M)}\eta (A,\alpha)}[r] & 
M A \ar^-{ \varepsilon U_{(M)} (A,\alpha)}[r] & U_{(M)} (A, \alpha)= A \\
MA \ar^-{\alpha}[u] \ar^-{\delta A}[r] \ar@{=}@/_1pc/[rr]& 
M^2 A \ar^-{M\alpha}[u] \ar^-{ \varepsilon M A}[r]& 
MA. \ar^-{\alpha}[u] }
$$
\end{proof}

For any firm Frobenius monad $M:\mathbb A \to \mathbb A$, the unit
$\eta:\mathbb A_{(M)}\to F_{(M)} U_{(M)}$ of the adjunction $U_{(M)}\dashv
F_{(M)}$ in Proposition \ref{prop:firm_Frob_adj} induces a functor $
K_{(M)}$ rendering commutative 
\begin{equation}\label{eq:K}
\xymatrix@C=60pt{
&\mathbb A^M\ar[d]^-{U^M}\\
\mathbb A_{(M)} \ar[ru]^-{K_{(M)}}\ar[r]^-{U_{(M)}}&
\mathbb A,}
\end{equation}
since $U_{(M)} F_{(M)} = M$ as comonads. Explicitly, $K_{(M)}$ is given
by 
$$
(A,MA\stackrel {\alpha\,} \to A)
\stackrel \phi \to 
(A',MA'\stackrel {\alpha'} \to A') \mapsto 
(A,A\stackrel{\eta (A,\alpha)} \longrightarrow MA)
\stackrel \phi \to
(A',A'\stackrel{\eta (A',\alpha')} \longrightarrow MA').
$$
 
\subsection{Non-unital Frobenius monads versus adjunctions to comodules}
In order to associate a non-unital Frobenius adjunction to any, not
necessarily firm, non-unital Frobenius monad, we shall work with the category 
of comodules instead of the categories of firm or non-unital modules
in the previous sections. 

Let $M:\mathbb A \to \mathbb A$ be a non-unital Frobenius monad. For any
$M$-comodule $(A,\alpha)$, throughout the paper the notation 
\begin{equation}\label{eq:alphabar}
\overline \alpha:=
\xymatrix@C=25pt{
MA\ar[r]^-{M\alpha}&
M^2A\ar[r]^-{\mu A}&
MA \ar[r]^-{\varepsilon A}&
A}
\end{equation}
will be used. 

\begin{lemma}\label{lem:aux}
Let $M:\mathbb A \to \mathbb A$ be a non-unital Frobenius monad. Using the
notation in \eqref{eq:alphabar}, for any $M$-comodule $(A,\alpha)$ the 
following assertions hold.
\begin{itemize}
\item[{(1)}]
The coaction $\alpha$ obeys $\mu A\cdot M\alpha=\alpha\cdot \overline
\alpha$. 
\item[{(2)}]
The identity $\alpha\cdot \overline \alpha= M \overline \alpha\cdot
\delta A$ holds. That is, $\overline\alpha$ is a morphism of $M$-comodules
$(MA, \delta A) \to (A, \alpha)$. 
\end{itemize} 
\end{lemma}

\begin{proof}
(1). The claim follows from the commutativity of the diagram
$$ \xymatrix@R=15pt@C=40pt{
M A \ar[d]_-{M \alpha} \ar[r]^-{M \alpha}&
M^2 A \ar[d]_-{M^2 \alpha} \ar[r]^-{\mu A}& 
MA \ar[r]^-{ \varepsilon A}\ar[d]^-{M\alpha}&
A \ar[dd]^-{\alpha}\\
M^2 A \ar[d]_-{\mu A}\ar[r]^-{M \delta A} &
M^3 A \ar[r]^-{\mu MA}&
M^2 A \ar[dr]^(.6){ \varepsilon M A}\\
MA \ar[rru]^(.4){\delta A} \ar@{=}[rrr] &&&
MA. }
$$

(2). By the naturality and the counitality of $\delta$ and the Frobenius
condition \eqref{Frob}, $M \overline \alpha\cdot \delta A=\mu A\cdot
M\alpha$. So the claim follows by part (1).
\end{proof}

\begin{proposition}\label{prop:nu_Frob_adj}
Any non-unital Frobenius monad $M:\mathbb A \to \mathbb A$ determines a
non-unital Frobenius adjunction $F^M U^M\to \mathbb A^M$ such that $U^M F^M=M$
as non-unital Frobenius monads. 
\end{proposition}

\begin{proof}
For any comonad $M$, $U^M \dashv F^M$ and $U^M F^M=M$ as comonads. A
non-unital adjunction $F^M U^M\to \mathbb A^M$ is provided by the $M$-comodule
morphisms $\overline \alpha: (MA,\delta A)= F^M U^M(A,\alpha)\to (A,\alpha)$
in Lemma \ref{lem:aux}~(2). In light of \eqref{eq:alphabar}, their naturality
follows by the naturality of $\mu$ and $\varepsilon$. The equality
$U^MF^M=M$ of non-unital monads follows by 
$\overline{\delta A}=
\varepsilon MA\cdot \mu MA\cdot M\delta A=
\varepsilon MA\cdot \delta A \cdot \mu A =
\mu A
$, cf. \eqref{eq:alphabar}.
\end{proof}

For any non-unital Frobenius monad $M:\mathbb A \to \mathbb A$, corresponding
to the non-unital adjunction $F^MU^M\to \mathbb A^M$ in Proposition
\ref{prop:nu_Frob_adj}, there is an induced functor $L_M:\mathbb A^M\to
\mathbb A_M$ as in Section \ref{sec:nu_adj}. It renders commutative the
diagram 
\begin{equation}\label{eq:V_M}
\xymatrix@C=60pt{
&\mathbb A_M\ar[d]^-{U_M}\\
\mathbb A^M \ar[ru]^-{L_M}\ar[r]^-{U^M}&
\mathbb A}
\end{equation}
and sends $(A,A\stackrel {\alpha\,} \to MA) \stackrel \phi \to (A',A'\stackrel
{\alpha'} \to MA')$ to $(A,\overline \alpha) \stackrel \phi \to
(A',\overline{\alpha}')$, cf. \eqref{eq:alphabar}. 

\section{Modules and comodules of a firm Frobenius monad}

The aim of this section is to see when the functor $K_{(M)}$,
associated in \eqref{eq:K} to a firm Frobenius monad $M$, is an
isomorphism. 

\begin{proposition} \label{prop:K_fullyfaithful}
For any firm Frobenius monad $M:\mathbb A \to \mathbb A$, the functor $
K_{(M)}$ in \eqref{eq:K} is fully faithful. 
\end{proposition}

\begin{proof}
For any firm $M$-module $(A, \alpha)$ and the functor $L_M$ in 
\eqref{eq:V_M}, 
$$
{L_M K_{(M)}} (A,\alpha) = 
(A, \xymatrix@C= 35pt{
MA \ar^-{M \eta (A, \alpha)}[r]& 
M^2 A \ar^-{\mu A}[r]& 
MA \ar^-{ \varepsilon A}[r]& 
A}).
$$
Since $\eta (A,\alpha)$ is a morphism of non-unital $M$-modules, and by one of
the triangle identities on the adjunction $U_{(M)}\dashv F_{(M)}$,
$\varepsilon A \cdot \mu A \cdot M \eta (A,\alpha) = \varepsilon A \cdot\eta
(A,\alpha)\cdot \alpha= \alpha$; that is, $L_M K_{(M)} =
J$. Since $J$ is faithful, we get that $K_{(M)}$ is faithful,
too. In order to see that $K_{(M)}$ is full, take a morphism 
$$
\xymatrix{K_{(M)} (A, \alpha) = (A, \eta (A,\alpha)) 
\ar^-{\phi}[r] & 
 K_{(M)} (A', \alpha) = (A', \eta (A',\alpha')) }
$$
in $\mathbb{A}^M$. Then 
$$
\xymatrix@C=40pt{
L_M K_{(M)} (A, \alpha) = (A, \alpha) \ar^-{ L_M 
\phi=\phi}[r] & 
 L_M K_{(M)} (A', \alpha) = (A', \alpha') }
$$
belongs to the full subcategory $\mathbb{A}_{(M)}$, and applying to it $
K_{(M)}$, we re-obtain $\phi$.
\end{proof}

\begin{theorem}\label{thm:iso}
For a non-unital Frobenius monad $M:\mathbb A \to \mathbb A$, the following
assertions are equivalent. 
\begin{itemize}
\item[{(1)}] 
The non-unital Frobenius adjunction $F^M U^M\to \mathbb A^M$ in
Proposition \ref{prop:nu_Frob_adj} is a firm Frobenius adjunction. 
That is, for any $M$-comodule $(A,\alpha)$, $\overline \alpha$ in
\eqref{eq:alphabar} is an epimorphism in $\mathbb A$ and there is a
coequalizer 
$$
\xymatrix@C=40pt{
M^2A\ar@<2pt>[r]^-{\mu A}\ar@<-2pt>[r]_-{M\overline{\alpha}}&
MA \ar[r]^-{\overline \alpha}&A}\qquad 
\textrm{in}\ \mathbb A_{M}.
$$
\item[{(2)}] 
$M$ is a firm Frobenius monad and the functor $K_{(M)}$ in \eqref{eq:K}
is an isomorphism. 
\end{itemize}
\end{theorem}

\begin{proof}
$(1) \Rightarrow (2).$ 
Since $M$ arises from a firm Frobenius adjunction in (1), it is a
firm Frobenius monad. 
Let $\mathbb{A}^M \stackrel {L_{(M)}\,} \longrightarrow
\mathbb{A}_{(M)} \stackrel J \to \mathbb{A}_M$ be a factorization of $
L_M$. We claim that $ L_{(M)}$ provides the inverse of $
K_{(M)}$. We know from the proof of Proposition \ref{prop:K_fullyfaithful}
that $J L_{(M)} K_{(M)} = L_M K_{(M)} = J$, so that
$L_{(M)} K_{(M)} = \mathbb{A}_{(M)}$. For any $M$-comodule
$(A,\alpha)$, $L_{(M)}(A,\alpha) = (A, \overline{\alpha})$
in \eqref{eq:alphabar} is firm by assumption. So $K_{(M)}
L_{(M)}(A,\alpha) = (A, \eta (A, \overline{\alpha}))$, and the proof is
complete if we show $\eta (A,\overline{\alpha}) =\alpha$. Since
$\overline{\alpha}$ is an epimorphism in $\mathbb A$, this follows by
commutativity of the diagram 
$$
\xymatrix@R=15pt@C=40pt{
A\ar[rrr]^-{\eta(A,\overline\alpha)}\ar[ddd]_-\alpha&&&
MA\ar@{=}[ddd]\\
& MA\ar[lu]^-{\overline\alpha}\ar[d]^-{M\alpha}\ar[r]^-{\delta A}&
M^2A\ar[ru]_-{M\overline\alpha}\ar[d]^-{M^2\alpha}\\
& M^2A\ar[r]^-{\delta MA}\ar[ld]_-{\mu A}&
M^3A\ar[d]^-{M\mu A}\\
MA\ar@{=}@/_1pc/[rrr]\ar[rr]^-{\delta A}&&
M^2 A \ar[r]^-{M\varepsilon A}&
MA}
$$
\medskip

\noindent
where the region on the left commutes by Lemma \ref{lem:aux}~{(1)}.

$(2) \Rightarrow (1)$ Since $L_M K_{(M)} = J$,
$L_M = J K_{(M)} ^{-1}$ is the desired factorization.
\end{proof}
 
Next we look for situations when the equivalent conditions in Theorem
\ref{thm:iso} hold. 

\begin{proposition}\label{prop:one_side}
For a non-unital Frobenius monad $M:\mathbb A \to \mathbb A$, assume that
there exists a natural section (i.e. right inverse) $\nu$ of the
multiplication $\mu:M^2\to M$ rendering commutative 
$$
\xymatrix@C=60pt@R=15pt{
M^2\ar[r]^-\mu\ar[d]_-{\nu M}&
M\ar[d]^-\nu\\
M^3\ar[r]^-{M\mu}&
M^2.}
$$ 
Then the equivalent conditions in Theorem \ref{thm:iso} hold. 
\end{proposition}

\begin{proof}
We will show that for every $M$-comodule $(A, \alpha)$, and
$\overline{\alpha}$ as in \eqref{eq:alphabar}, the diagram 
$$
\xymatrix@C=40pt{
M^2 A \ar^-{\mu A}@<2pt>[r] \ar@<-2pt>_-{M \overline{\alpha}}[r]& 
MA \ar^-{\overline{\alpha}}[r] \ar@{-->}@/_15pt/_{\nu A}[l] & 
\ar@{-->}@/_15pt/_{M \varepsilon A \cdot \nu A \cdot \alpha }[l] A }
$$
is a contractible coequalizer in $\mathbb A$ (in the sense of \cite[page
93]{TTT}). By assumption, $\mu A \cdot \nu A = MA$. By Lemma
\ref{lem:aux}~{(1)}, and naturality of $\nu$, we get that the diagram 
$$
\xymatrix@C=40pt@R=12pt{
MA \ar@{=}[r]\ar^-{\overline{\alpha}}[dd] & 
MA \ar^-{\nu A}[r] \ar^-{M \alpha}[d] & 
M^2 A \ar@{=}[r] \ar^-{M^2 \alpha}[d] & 
M^2A \ar^-{M \overline{\alpha}}[dd] \\
& M^2 A \ar^-{\nu MA}[r] \ar^-{\mu A}[d] & 
M^3 A \ar^-{M \mu A}[d] & \\
A \ar^-{\alpha}[r] & 
MA \ar^-{\nu A}[r] & 
M^2A \ar^-{M \varepsilon A}[r] & 
MA }
$$
is commutative. Thus, $M\overline{\alpha} \cdot \nu A = M \varepsilon A \cdot
\nu A \cdot \alpha \cdot \overline{\alpha}$. Using that $\nu$ is natural, the
Frobenius condition and the assumption, also the following diagram is seen to
commute. 
$$
\xymatrix@C=40pt@R=12pt{
M \ar_-{\nu}[ddd] \ar^-{\delta}[rrr] & & & 
M^2 \ar^-{\nu M}[ddd]\\
& M^2 \ar_-{\nu M}[d] \ar^-{\mu}[lu] \ar^-{M \delta}[r] & 
M^3 \ar_-{\mu M}[ru] \ar^-{\nu M^2}[d] & \\
& M^3 \ar_-{M \mu}[dl] \ar^-{M^2 \delta}[r] & 
M^4 \ar^-{M \mu M}[dr] & \\
M^2 \ar^{M \delta}[rrr] & & & 
M^3}
$$
Since $\mu$ is a (split) natural epimorphism by assumption, also the outer
rectangle commutes, what implies commutativity of 
$$
\xymatrix@C=40pt@R=12pt{
A\ar^-{\alpha}[d] \ar^-{\alpha}[r] & 
MA \ar^-{\nu A}[r] \ar_-{M \alpha}[d] & 
M^2A \ar^-{M \varepsilon A}[r] & 
MA \ar_-{M \alpha}[d] \\
MA \ar^-{\delta A}[r] \ar@{=}[d] & 
M^2A \ar^-{\nu MA}[r] & 
M^3 A \ar^-{M \varepsilon M A}[r] & 
M^2A .\\
MA \ar^-{\nu A}[rr]& & 
M^2A \ar^-{M \delta A}[u]\ar@{=}[ur]}
$$
Composing both equal paths around this diagram by $\varepsilon A\cdot \mu A$,
using that $\nu$ is a section of $\mu$ and the counitality of $\alpha$, we
obtain $\overline \alpha\cdot M\varepsilon A\cdot \nu A\cdot \alpha =A$. Since
in this way $\overline \alpha$ is a split epimorphism, it is taken by $M$ to a
(split) epimorphism. So we conclude by Lemma \ref{lem:Malpha} that
$L_M (A,\alpha)=(A,\overline\alpha)$ is a firm $M$-module. 
\end{proof}

Every Frobenius pair of functors in the sense of \cite{Castano/alt:1999} gives
obviously a (unital) Frobenius adjunction and, therefore, a (unital) Frobenius
monad. A more interesting situation is described in the following
corollary. Separable functors were introduced and studied in 
\cite{Nastasescu/alt:1989}. 

\begin{corollary}\label{cor:separablemodcomod}
Let $U:\mathbb B \to \mathbb A$ be a separable functor possessing a right
adjoint $F$. Then $UF$ carries the structure of a firm Frobenius monad such
that the comparison functor $K_{(UF)}: \mathbb{A}_{(UF)} \to \mathbb{A}^{UF}$
is an isomorphism. 
\end{corollary}

\begin{proof}
By Rafael's theorem \cite{Rafael:1990}, there exists a retraction (i.e. left
inverse) $\varphi$ of the unit $\eta:\mathbb B \to FU$ of the adjunction. Then
$\varphi:FU\to \mathbb B$ is a non-unital Frobenius adjunction so that $UF$ is
a non-unital Frobenius monad. The claim follows by applying to it Proposition
\ref{prop:one_side}, putting $\nu:=U\eta F$. 
\end{proof}

\section{Application: Firm Frobenius algebras over commutative rings}

\subsection{Firm Frobenius algebra}
Let $k$ be an associative, unital, commutative ring and denote the category of
$k$-modules by $\mathbb M_k$. It is a monoidal category via the $k$-module
tensor product $\ox$ and the neutral object $k$. 

Any associative algebra $R$ --- possibly without a unit --- over $k$ may be
equivalently defined as a non-unital monad $(-) \ox R: \mathbb{M}_{ k} \to
\mathbb{M}_{ k}$. The category of non-unital modules for this monad ---
equivalently, the category of non-unital modules for the algebra $R$ --- will
be denoted by $\mathbb M_R$. For any (non-unital) right $R$-module
$(A,\alpha)$ and left $R$-module $(B,\beta)$, we denote by $A\ox_R B$ the
coequalizer of $\alpha\ox B$ and $A\ox \beta$ in $\mathbb M_k$ and we call it
the {\em $R$-module tensor product}. 

\begin{proposition}
For a non-unital $k$-algebra $R$ and a non-unital right $R$-module $(A,A\ox R
\stackrel {\alpha\,} \to A)$, the following assertions are equivalent. 
\begin{itemize}
\item[{(1)}] 
$(A,\alpha)$ is a firm module for the non-unital monad $(-)\ox R:\mathbb M_k
\to \mathbb M_k$. 
\item[{(2)}] 
The action $\alpha:A\ox R\to A$ projects to a bijection $A \ox_R R \to A$. 
\end{itemize}
\end{proposition}

\begin{proof}
Since $(-) \ox R$ is a right exact endofunctor on $\mathbb{M}_k$, we get that
coequalizers exist in $\mathbb{M}_R$ and the forgetful functor $\mathbb M_R
\to \mathbb M_k$ creates them. Hence \eqref{eq:firm_module} is a coequalizer
in $\mathbb{M}_R$ if and only if it is a coequalizer in $\mathbb{M}_k$ so if
and only if (2) holds. This proves $(1)\Rightarrow (2)$. If (2) holds then
$\alpha$ is surjective, hence it is an epimorphism in $\mathbb{M}_k$ proving
$(2)\Rightarrow (1)$. 
\end{proof}

Similarly, $(-) \ox R:\mathbb M_k \to \mathbb M_k$ is a firm monad if and only
if the multiplication map $R \ox R \to R$ projects to an isomorphism $R \ox_{
R} R \to R$. That is, if and only if $R$ is a firm ring in the sense of
\cite{Qui}. 

By a {\em non-unital Frobenius $k$-algebra} we mean a $k$-module $R$ such that
$(-)\ox R:\mathbb M_k\to \mathbb M_k$ is a non-unital Frobenius
monad. Explicitly, this means that $R$ is a non-unital $k$-algebra with
multiplication $\mu:R\ox R \to R$ and a $k$-coalgebra with comultiplication
$\Delta:R\to R\ox R$ and counit $\epsilon:R\to k$, such that $\mu$ is a
morphism of $R$-bicomodules, equivalently, $\Delta$ is a morphism of
$R$-bimodules, that is, the following diagram commutes. 
$$
\xymatrix@R=10pt{
R\ox R\ar[rr]^-{R\ox \Delta}\ar[dd]_-{\Delta \ox R}\ar[rd]^-{\mu}&&
R\ox R \ox R\ar[dd]^-{\mu \ox R}\\
& R \ar[rd]^-\Delta\\
R\ox R \ox R\ar[rr]^-{R \ox \mu}&&
R\ox R}
$$
A {\em firm Frobenius $k$-algebra} is a non-unital Frobenius $k$-algebra which
is a firm algebra. 

\subsection{The Casimir multiplier}
In the case of a unital Frobenius algebra $R$, the $R$-bilinear
comultiplication is tightly linked to a so-called Casimir element in $R \ox R$
(the image under $\Delta$ of the unit $1$). In our present setting, this is
only possible if we allow the Casimir element to be a multiplier
\cite{Dauns:1969}. 

Let $R$ be a non-unital $k$-algebra with a non-degenerate multiplication. That
is, assume that any of the conditions $(sr=0,\ \forall s\in R)$ and
$(rs=0,\ \forall s\in R)$ implies $r=0$. A {\em multiplier} on $ R $ is a pair 
$(\lambda,\varrho)$ of $k$-module endomorphisms of $R$ such that 
\begin{equation}\label{eq:multiplier}
{ \varrho}(r)s = 
 r { \lambda}(s),\qquad \textrm{for all\ }
 r,s \in R . 
\end{equation}
By \cite[1.5]{Dauns:1969}, $\lambda$ is a right $R$-module map and $\rho$ is a
left $R$-module map. 

The $k$-module $\mathcal{M}( R )$ of all multipliers is a unital associative
algebra with the { multiplication $(\lambda,\varrho) (\lambda',\varrho') =
(\lambda \lambda', \varrho' \varrho)$, where juxtaposition in the components
means composition of maps. Throughout, we denote by $1$ the unit element
$(\mathsf{id},\mathsf{id})$ of $\M R$. There exists an injective
homomorphism of algebras from $R$ to $\M R$ sending $r \in R$ to the
multiplier $(\lambda_{r}, \rho_{r})$, where $\lambda_{r}(s) = rs$ and
$\rho_{r}(s) = sr$, for all $r,s \in R $. The image of $R$ becomes a two-sided
ideal of $\mathcal{M}(R)$: a multiplier $\omega = (\lambda, \varrho)$ acts on
an element $r \in R$ by $\omega r = \lambda(r)$, and $r\omega = \varrho(r)$
(so that \eqref{eq:multiplier} can be rewritten as $(r\omega )s=r(\omega s)$,
for all $s,r\in R$ and $\omega \in \M R$; allowing for a simplified writing
$r\omega s$). This yields inclusions $R\ox R\subseteq \M R \ox \M R \subseteq
\M {R\ox R}$. By the non-degeneracy of the multiplication of $R$, two
multipliers $\omega$ and $\omega'$ on $R$ are equal if and only if $\omega
r=\omega' r$ for all $r\in R$ and if and only if $r \omega=r \omega'$ for all
$r\in R$. 

\begin{proposition}
{ Let $R$ be a firm algebra with non-degenerate multiplication over a
commutative ring $k$. Then} $R$ is a firm Frobenius algebra if and only if
there exists a multiplier $e \in \M{R \ox R}$ and a linear map $\epsilon : R
\to k$ such that, for all $r\in R$, $(r \ox 1)e = e(1 \ox r)$ is an element of
$R\ox R$ and 
$$
(\epsilon \ox R)((r \ox 1)e)=r=
(R\ox \epsilon)(e(1 \ox r)).
$$
\end{proposition}

\begin{proof}
Given an $R$-bilinear comultiplication $\Delta : R \to R \ox R$, define
$\tilde{\Delta} : \M{R} \to \M{R \ox R}$ by 
\begin{eqnarray}\label{eq:deltatilde}
&&\tilde{\Delta}(\omega)(s \ox r)= \Delta(\omega r)(s \ox 1) \\
&&(s \ox r)\tilde{\Delta}(\omega)= (1 \ox r) \Delta (s\omega),
\quad \textrm{for } \omega \in \M{R},\ s, r \in R.
\nonumber
\end{eqnarray}
The following computation shows that $\tilde{\Delta}(\omega)$ is a multiplier: 
\begin{eqnarray*}
((s' \ox r') \tilde{\Delta}(\omega))(s \ox r) & = & 
(1 \ox r')\Delta(s'\omega)(s \ox r) 
= (1 \ox r')\Delta(s'\omega r)(s \ox 1) \\
& = & (s' \ox r')\Delta (\omega r)(s \ox 1) 
= (s' \ox r')(\tilde{\Delta}(\omega)(s \ox r)),
\end{eqnarray*}
where in the second equality we used that $\Delta$ is right $R$-linear, and in
the third one that it is left $R$-linear. Put $e = \tilde{\Delta}(1) \in \M{R
\ox R}$. Take $s' \ox r' \in R \ox R$ and $r \in R$. Using once more that
$\Delta$ is left and right $R$-linear, 
$$
(r \ox 1)e(s' \ox r') = 
(r \ox 1)\Delta(r')(s' \ox 1) = 
\Delta(rr')(s' \ox 1) = 
\Delta(r)(s' \ox r').
$$
Therefore, $(r \ox 1) e = \Delta (r)$ and so $(\epsilon \ox R)((r\ox 1)e) =
r$. Symmetrically, $e(1 \ox r) = \Delta (r)$ and so $(R \ox \epsilon)(e(1 \ox
r)) = r$. 

Conversely, assume the existence of $e \in \M{R \ox R}$ and $\epsilon : R \to
k$ as in the claim. Define 
$$
\Delta(r) = (r \ox 1)e = e (1 \ox r) \in R \ox R,\quad \hbox{ for all } r \in
R.
$$
This is clearly an $R$-bilinear comultiplication with the counit
$\epsilon$. It remains to prove that $\Delta$ is coassociative. Since $R$ is
firm, we know that $\mu$ is an epimorphism. Therefore, the coassociativity of
$\Delta$ follows from the commutativity of the diagram 
$$
\xymatrix@C=40pt@R=10pt{R \ar^-{\Delta}[rrr] \ar_-{\Delta}[ddd] & & & 
R \ox R \ar^-{\Delta \ox R}[ddd]\\
& R \ox R \ar_-{\Delta \ox R}[d] \ar^-{R \ox \Delta}[r] \ar^-{\mu}[lu] & 
R \ox R \ox R \ar_-{\mu \ox R}[ru] \ar^-{\Delta \ox R \ox R}[d] & \\
& R \ox R \ox R \ar_-{R \ox \mu}[dl] \ar^-{R \ox R \ox \Delta}[r] & 
R \ox R \ox R \ox R \ar^-{R \ox \mu \ox R}[rd]& \\
R \ox R \ar^-{R \ox \Delta}[rrr] & & & R \ox R \ox R.}
$$
\end{proof}

\subsection{Firm modules and comodules}
The following extension of Abrams' classical theorem on unital Frobenius
algebras \cite{LAbrams} is an immediate consequence of Theorem \ref{thm:iso}. 

\begin{theorem} \label{thm:iso_k}
Let $R$ be a non-unital Frobenius algebra over a commutative ring $k$. Then
the following assertions are equivalent. 
\begin{itemize}
\item[{(1)}] 
Any right $R$-comodule $N$ is a firm right $R$-module via the action $n\cdot
r:=n_0\epsilon(n_1r)$ (where Sweedler's implicit summation index notation
$n\mapsto n_0\ox n_1$ is used for the coaction). 
\item[{(2)}] 
$R$ is a firm Frobenius $k$-algebra and the category $\mathbb M_{(R)}$ of firm
right $R$-modules and the category $\mathbb M^R$ of right $R$-comodules are
isomorphic via the following mutually inverse functors. The functor $\mathbb
M_{(R)} \to \mathbb M^R$ sends a firm right $R$-module $M$ to 
$$
(M, \xymatrix@C=30pt{
M\ar[r]^-\cong&
M\ox_R R \ar[r]^-{M\ox_R \Delta}&
M\ox_R R \ox R \ar[r]^-\cong&
M\ox R}).
$$
The functor $\mathbb M^R \to \mathbb M_{(R)}$ sends an $R$-comodule $(N,\rho)$
to 
$$
(N,\xymatrix{
N\ox R \ar[r]^-{\rho\ox R}&
N\ox R \ox R \ar[r]^-{N\ox \mu}&
N\ox R \ar[r]^-{N\ox \epsilon}&
N}).
$$
On the morphisms both functors act as the identity maps.
\end{itemize}
\end{theorem} 

\subsection{Example: coseparable coalgebra}\label{sec:cosep}
A (coassociative and counital) coalgebra $C$ over a commutative ring $k$ is
said to be {\em coseparable} if there is a $C$-bicomodule retraction
(i.e. left inverse) of the comultiplication. Equivalently, the forgetful
functor $U$ from the category of (say, right) $C$-comodules $\mathbb{M}^C$ to
$\mathbb{M}_k$ is separable \cite[Corollary 3.6]{Brzezinski:2002}. Since $U$
is left adjoint to $(-) \ox_k C$, Corollary \ref{cor:separablemodcomod} yields
a firm Frobenius monad $(-) \ox_k C$ on $\mathbb{M}_k$. Therefore $C$ 
admits the structure of a firm ring. The multiplication is given by the
bicolinear retraction of the comultiplication (see \cite{BKW} for a direct
proof). By Corollary \ref{cor:separablemodcomod}, the category of firm modules
$\mathbb{M}_{(C)}$ and the category of comodules $\mathbb{M}^C$ are isomorphic. 

The above reasoning can be repeated for a coring $C$ over an arbitrary
(associative and unital) algebra $A$. By \cite[Corollary
3.6]{Brzezinski:2002}, $C$ is a coseparable coring if and only if the (left
adjoint) forgetful functor from the category of (say, right) $C$-comodules to
the category of (right) $A$-modules is separable. So by Corollary
\ref{cor:separablemodcomod}, $C$ possesses a firm ring structure, described
already in \cite{BKW}. In this way Corollary \ref{cor:separablemodcomod}
extends \cite[Proposition 2.17]{BoVe}. 

\subsection{Example: graded rings}
Let $G$ be an arbitrary group. For any commutative ring $k$, consider the
$k$-coalgebra $R$ with free $k$-basis $\{ p_g : g \in G \}$, comultiplication
$\Delta (p_g) = p_g \ox p_g$ and counit $\epsilon (p_g) = 1$ for $g \in G$.
This is in fact a coseparable coalgebra via the bicomodule section of $\Delta$
--- hence associative multiplication --- determined by $p_g \ox p_h \mapsto
p_g$ if $g=h$ and $p_g \ox p_h \mapsto 0$ otherwise, for $g, h \in G$. 

For any $G$-graded unital $k$-algebra $A = \oplus_{g \in G} A_g$, the linear
map 
$$
R \ox A \to A \ox R, \qquad
p_g \ox a_h \mapsto a_h \ox p_{gh}, 
\qquad \textrm{for}\ g, h \in G\ \textrm{and}\ a_h \in A_h,
$$ 
is an entwining structure (a.k.a. mixed distributive law) and, therefore, it
defines an $A$-coring structure on $A \ox R$ (see \cite[Proposition
2.2]{Brzezinski:2002}). The $A$-bimodule structure on $A \ox R$ is determined
by 
$$
c(a \ox p_g)b_h = cab_h \ox p_{gh}, 
\qquad \textrm{for}\ a,c \in A,\quad b_h \in A_h,\quad g,h \in G,
$$ 
and its comultiplication and counit are the linear extensions of $a \ox p_g
\mapsto (a \ox p_g) \ox_A (1 \ox p_g)$ and $a\ox p_g \mapsto a$,
respectively. The category of right comodules over this coring is isomorphic
to the category of entwined (or mixed) modules \cite[Proposition
2.2]{Brzezinski:2002} which, in the present situation, is isomorphic to the
category $\mathsf{gr}$-$A$ of $G$-graded right $A$-modules. 

Thanks to the coseparability of the coalgebra $R$, the $A$-coring $A \ox R$ is
coseparable: its comultiplication has a bicomodule section 
$$
\mu : (A \ox R) \ox_A (A \ox R) \to A \ox R,\qquad
(a \ox p_g) \ox_A (b \ox p_h) \mapsto a b_{g^{-1}h} \ox p_h,
$$ 
for $a, b \in A, g, h \in G$. Therefore, the discussion at the end of Example
\ref{sec:cosep} shows that $A \ox R$ is a firm $k$-algebra via the
multiplication induced by $\mu$. (In fact, this firm algebra even has local
units.) Moreover, the category $\mathsf{gr}$-$A$ is isomorphic to the category
$\mathbb{M}_{(A \ox R)}$ of firm modules. Since the firm algebra $A\ox R$ is
isomorphic to the smash product $A \sharp G^*$ in \cite{Beattie:1988}, this
isomorphism $\mathbb{M}_{(A \ox R)}\cong \mathbb{M}_{(A \sharp G^*)}\cong
\mathsf{gr}$-$A$ reproduces the main result of \cite{Beattie:1988}. 

\subsection{Example: Frobenius algebra with local units}
We say that an algebra $R$ is an algebra with {\em local units} if there is a
set $\mathsf{E}$ of idempotent elements in $R$ such that for every finite set
$r_1, \dots, r_n \in R$ there is $e \in \mathsf{E}$ obeying $er_i = r_i=r_ie$
for every $i = 1, \dots, n$. (This definition of local units is the one used
in \cite{ElKaoutit/Gomez:2012}, and it can be traced back to
\cite{Abrams:1983} and \cite{Anh/Marki:1983}. It is more general than
\cite[Definition 1.1]{Abrams:1983}, since we do not assume that the elements
of $\mathsf{E}$ commute. In fact, the present notion generalizes that of
\cite{Abrams:1983} since, when the idempotents of $\mathsf{E}$ commute, it is
enough to require that for each element $r \in R$ there exists $e \in
\mathsf{E}$ such that $er = r= re$, see \cite[Lemma 1.2]{Abrams:1983}.) Every
algebra $R$ with local units is firm, and a right $R$-module $M$ is firm if
and only if for every $m \in M$ there is $e \in \mathsf{E}$ such that $m\cdot
e= m$. 

\begin{corollary}\label{slu}
Let $R$ be an algebra with local units over a commutative ring. If $R$ is a
firm Frobenius algebra, then the category of firm right $R$-modules and the
category of right $R$-comodules are isomorphic. 
\end{corollary}

\begin{proof}
According to Theorem \ref{thm:iso_k}, we need to prove that any right
$R$-comodule $M$ is a firm right $R$-module via the action $m\cdot
r:=m_0\epsilon(m_1r)$: Given $m \in M$, we know that $m_0\ox m_1 = \sum_i m_i
\ox r_i$ for finitely many $m_i \in M$, $r_i \in R$. Let $e \in \mathsf{E}$
such that $r_i e = r_i$ for every $i$. Then $m \cdot e= \sum_i m_i\epsilon
(r_i e) = \sum_i m_i \epsilon (r_i) = m$. 
\end{proof}

\subsection{Example: co-Frobenius coalgebra over a field}
Let $C$ be a coalgebra over a field $k$. The dual vector space $C^*$ is an
associative and unital $k$-algebra via the convolution product $(\phi \ast
\psi)(c) = \phi (c_1) \psi (c_2)$, for all $c \in C$ and $\phi, \psi \in C^*$,
where for the comultiplication the Sweedler-Heynemann index notation is used,
implicit summation understood. Every right/left $C$-comodule becomes then a
left/right $C^*$-module, in particular, $C$ becomes a $C^*$-bimodule. The
coalgebra $C$ is said to be {\em left/right co-Frobenius} if there exists a
monomorphism of left/right $C^*$-modules $C\to C^*$, see \cite{Lin:1977}. 
 
\begin{proposition}\label{prop:coFrob}
For a coalgebra $C$ over a field $k$, the following assertions are equivalent.
\begin{itemize}
\item[{(1)}] 
$C$ is a left and right co-Frobenius coalgebra.
\item[{(2)}] 
The given coalgebra structure of $C$ extends to a non-unital Frobenius algebra
with a non-degenerate multiplication. 
\item[{(3)}] 
The given coalgebra structure of $C$ extends to a firm Frobenius algebra whose
multiplication admits local units. 
\end{itemize}
\end{proposition}

\begin{proof}
$(1)\Rightarrow (3)$ is proved, in fact, in \cite{CaDaNa}. The main line of
the reasoning can be summarized as follows. When $C$ is semiperfect (that
is, the categories of left and right $C$-comodules have enough projectives),
then the left and right rational ideals of the convolution algebra $C^*$
coincide \cite[Corollary 3.2.16]{Dascalescu/alt:2000}. Let
$\mathsf{Rat}(C^*)$ denote their common value. By \cite[Corollary
3.2.17]{Dascalescu/alt:2000}, $\mathsf{Rat}(C^*)$ is a ring with local
units. Now, if $C$ is left and right co-Frobenius, then $C$ is semiperfect
\cite{Lin:1977} and, by \cite[Theorem 2.1]{Gomez/alt:2003}, there is an
isomorphism of say, right $C^*$-modules $C \cong \mathsf{Rat}(C^*)$. Then we
may pull-back the multiplication of $\mathsf{Rat}(C^*)$ to $C$ so that $C$
becomes a $k$-algebra with local units. By \cite[Theorem 2.2]{CaDaNa}, the
resulting multiplication is a morphism of $C^*$-bimodules. Using the relation
between the $C^*$-actions and the $C$-coactions on $C$, we conclude that the
opposite of this multiplication is a $C$-bicomodule map. Hence $C$ is also a
firm Frobenius algebra. 
 
$(3)\Rightarrow (2)$ is trivial.

$(2)\Rightarrow (1)$. A right $C^*$-module map $C\to C^*$ is provided by
$c\mapsto \epsilon(c-)$. Indeed, for $c,d\in C$ and $\varphi\in C^*$, 
$$ 
\epsilon((c\leftharpoonup \varphi)d)=
\varphi(c_1)\epsilon(c_2d)=
\varphi(cd)=
\epsilon(cd_1)\varphi(d_2)=
(\epsilon(c-)\varphi)(d).
$$
It is injective since if $\epsilon(cd)=0$ for all $d\in C$, then
$0=\epsilon(cd_1)d_2=cd$, hence $c=0$ by the non-degeneracy of the
multiplication. Symmetrically, a monomorphism of left $C^*$-modules is
provided by $c\mapsto \epsilon(-c)$. 
\end{proof}

Note that the right $C^*$-module map $C\to C^*$, $c\mapsto \epsilon(c-)$ in
the proof of Proposition \ref{prop:coFrob} is anti-multiplicative by 
$$
\epsilon(cd_1)\epsilon(c'd_2)=
\epsilon(c'\epsilon(cd_1)d_2)=
\epsilon(c'cd),
\qquad \textrm{for all } c,c',d\in C.
$$
So we conclude by Corollary \ref{slu} that if for a coalgebra $C$ over a field
the equivalent assertions in Proposition \ref{prop:coFrob} hold, then the
categories $\mathbb{M}^C$ and
${}_{(\mathsf{Rat}(C^*))}\mathbb{M}\cong\mathbb{M}_{(C)}$ are
isomorphic. Therefore, Corollary \ref{slu} extends \cite[Theorem 2.3]{CaDaNa}
and \cite[Proposition 2.7]{Beattie/alt:1998}. 

Non-degenerate algebras over a field equipped with a so-called separability
idempotent, were discussed recently in \cite{VDae}. Using the terminology of
the current paper, they are in fact coalgebras obeying the equivalent
properties in Proposition \ref{prop:coFrob} and the additional requirement
that their comultiplication splits the multiplication (`separability' of the
Frobenius structure); cf. Section \ref{sec:cosep}. In particular, they have
local units, answering in affirmative an open question in \cite{VDeWa}. 
\medskip

\noindent\textbf{Acknowledgement.}\textit{Research partially supported by the
Hungarian Scientific Research Fund OTKA, grant no. K 108384 and the 
Spanish Ministerio de Ciencia en Innovaci\'on and the European Union, grant
MTM2010-20940-C02-01. 
\\ 
JG-T thanks the members of the Wigner RCP for their kind invitation and for
the warmest hospitality during his visit in February of 2013. GB thanks
Alfons Van Daele for highly enlightening and pleasant discussions on the
subject.}

%
\end{document}